\documentclass[12pt]{article}
\usepackage[T1]{fontenc}
\usepackage{mathtools,amsthm,amssymb}
\usepackage[cm]{fullpage}
\usepackage[dvipsnames]{xcolor}
\usepackage[colorlinks=true,linkcolor=ForestGreen,citecolor=ForestGreen,urlcolor=Blue]{hyperref}
\usepackage{cleveref}
\usepackage{crossreftools}
\usepackage{arydshln}
\usepackage{tikz-cd}
\usepackage{sagetex}
\usepackage{cite}
\makeatletter
\newcommand{\citecomment}[2][]{\citen{#2}#1\citevar}
\newcommand{\citeone}[1]{\citecomment{#1}}
\newcommand{\citetwo}[2][]{\citecomment[,~#1]{#2}}
\newcommand{\citevar}{\@ifnextchar\bgroup{;~\citeone}{\@ifnextchar[{;~\citetwo}{]}}}
\newcommand{\citefirst}{\@ifnextchar\bgroup{\citeone}{\@ifnextchar[{\citetwo}{]}}}
\newcommand{\cites}{[\citefirst}
\makeatother

\newcommand{\seqnum}[1]{\href{https://oeis.org/#1}{\rm \underline{#1}}}

\DeclareMathOperator{\id}{id}
\DeclareMathOperator{\lcm}{lcm}
\DeclareMathOperator{\Mag}{Mag}
\DeclareMathOperator{\kMag}{\mathnormal{k}-Mag}

\newtheorem{theorem}{Theorem}
\newtheorem{lemma}[theorem]{Lemma}

\theoremstyle{definition}
\newtheorem{definition}[theorem]{Definition}
\newtheorem{example}[theorem]{Example}
\newtheorem{problem}[theorem]{Problem}

\theoremstyle{remark}
\newtheorem*{remark}{Remark}

\title{\bf Counting Finite Magmas}
\date{}
\author{
Philip Ture\v{c}ek\\
Department Mathematik\thanks{I am no longer a member of this department.}\\
Friedrich-Alexander-Universit\"at\\
Cauerstra{\ss}e 11\\
91058 Erlangen\\
Germany\\
\href{mailto:philip.turecek@fau.de}{\tt philip.turecek@fau.de}
}

\begin{document}
	\maketitle
	\begin{abstract}
		Given a non-negative integer $n$, we establish a formula for the number of finite magmas on a set with cardinality $n$ up to isomorphism. We then generalize the method to operations with arbitrary finite arity, which yields a corrected version of Harrison's formula. Moreover, we present the cycle index as a helpful tool for practical computations and, based on that, we give a suitable code in Sage with a few generated examples.
	\end{abstract}
	
	\section{Introduction}
	While algebraic structures such as fields or Abelian groups are quite easy to classify in the finite case, a classification becomes more challenging if one drops certain conditions, such as the existence of inverse elements, the commutativity, or the associativity. On the other hand, by dropping all conditions, it is sometimes possible to compute the number of isomorphy classes depending on the cardinality of the underlying set. Since the most common algebraic structures have binary operations, we will first focus on those with one such operation.
	\begin{definition}
		A \emph{magma} is a set $M$ equipped with a map $\ast \colon M \times M \to M$, also called (internal binary) \emph{operation} on $M$, which we formally write together as a pair $(M,\ast)$.
	\end{definition}

	Throughout the paper, let $n$ be a non-negative integer and $[n] \coloneqq \{1,\dots,n\}$. Moreover, we consider $([n],\ast)$ as a prototype for a finite magma on a set with cardinality $n$. Note that for $n=0$, there is the empty map as the only possible operation. If one wants to write $\ast$ explicitly, one can do this via an operation table, also called Cayley table:
	\[
	\begin{array}{c|ccc}
		\ast & 1 & \cdots & n \\ \hline
		1 & 1\ast1 & \cdots & 1\ast n \\ 
		\vdots & \vdots & \ddots & \vdots \\
		n & n\ast1 & \cdots & n\ast n \\
	\end{array}.
	\]
	A natural question that arises is how many different operations exist on $[n]$. The answer is easy: For each of the $n^2$ entries of the Cayley table, there are $n$ different possibilities to assign a value, and thus, there are $n^{n^2}$ different operations in total. We set $0^0\coloneqq1$ in order to be consistent.

	Although binary operations, i.e., Cayley tables, can be formally different, it might be possible to obtain one such table by renaming the elements of some other, and vice versa.
	\begin{example}\label{ex:tables}
		Consider the Cayley tables
		\[
		\begin{array}{c|cc}
			\ast & 1 & 2 \\ \hline
			1 & 1 & 1 \\ 
			2 & 1 & 2 \\
		\end{array}
		\quad\text{and}\quad
		\begin{array}{c|cc}
			\star & 1 & 2 \\ \hline
			1 & 1 & 2 \\ 
			2 & 2 & 2 \\
		\end{array},
		\]
		which are, of course, formally different, but by interchanging the elements $1$ and $2$ in the first table, we get
		\[
		\begin{array}{c|cc}
			\ast' & 2 & 1 \\ \hline
			2 & 2 & 2 \\ 
			1 & 2 & 1 \\
		\end{array},
		\]
		which, after rearranging, is obviously the same as the second table. In the same way we also obtain the first table from the second. This means that the two initial tables are structurally the same.
	\end{example}
	Instead of asking how many different operations exist on $[n]$ in total, we can ask how many \emph{structurally} different operations exist on $[n]$. To put it another way:
	\begin{problem}
		How many structurally different magmas exist on a set with cardinality $n$?
	\end{problem}
	\Cref{thm:main}, which goes back to Harrison \cite[Theorem 4]{H}, provides an answer; the corresponding sequence is \seqnum{A001329}. Harrison also gives generalizations for finite algebraic structures with an operation of arbitrary finite arity $k$, which we will call \emph{$k$-magmas}. However, his corresponding formula \cite[Theorem 5]{H} contains a typing error concerning a factorial symbol as well as a more serious error concerning an exponent. For that reason, we will deduce a corrected version in \Cref{sec:gen}; see \Cref{thm:main_k}. This means that sequence \seqnum{A001331} based on Harrison's formula is incorrect, while sequence \seqnum{A091510} describing the same (namely $3$-magmas) is correct.
	\begin{theorem}[Harrison \cite{H}]\label{thm:main}
		There are 
		\[
		\sum_{j \in J_n} \frac{1}{\prod_{i=1}^n i^{j_i} j_i!} \prod_{r,s=1}^n \bigg(\sum_{d\mid\lcm(r,s)} d j_d\bigg)^{\gcd(r,s)j_rj_s}
		\]
		pairwise non-isomorphic magmas on a set with cardinality $n$, where $J_n$ denotes the set of all sequences $j=(j_i)_{i=1}^{\infty}$ of non-negative integers such that $\sum_{i=1}^{\infty} ij_i = n$.
	\end{theorem}
	
	\section{Proof of \Cref{thm:main}}\label{sec:pot}
	 We first formalize what we did in \Cref{ex:tables}.
	\begin{definition}\label{def:morphisms}
		Given two magmas $(M,\ast)$ and $(N,\star)$, a \emph{magma homomorphism} is a map $\phi \colon M \to N$ such that $\phi(a\ast a')=\phi(a)\star\phi(a')$ for all $a,a' \in M$. A \emph{magma isomorphism} is a bijective magma homomorphism. If there exists a magma isomorphism $M \to N$, we say that $M$ and $N$ are \emph{isomorphic}, which we write as $M \cong N$, or to be more precise, as $(M,\ast)\cong(N,\star)$; if we additionally have $(M,\ast)=(N,\star)$, an isomorphism is also called \emph{automorphism}.
	\end{definition}
	\begin{remark}
		A bijective magma homomorphism $\phi\colon M \to N$ already implies that its inverse map $\phi^{-1}\colon N \to M$ is also a homomorphism and thus, by being bijective as well, an isomorphism of magmas. To see this, let $b,b' \in N$, and let $a,a'\in M$ be the respective preimages under $\phi$. Then we have
		\[
		\phi^{-1}(b\star b') = \phi^{-1}(\phi(a)\star\phi(a')) = \phi^{-1}(\phi(a\ast a')) = a\ast a' = \phi^{-1}(b) \ast \phi^{-1}(b').
		\]
		Thus, there exists a magma isomorphism $M\to N$ if and only if there exists a magma isomorphism $N\to M$, which justifies the symmetric notation $M\cong N$.
	\end{remark}

	When are two finite magmas $([n],\ast)$ and $([m], \star)$ isomorphic? By definition, this is the case if and only if there is a bijective map $\sigma \colon [n] \to [m]$ such that
	\begin{equation}\label{eqn:iso}
	\sigma(a\ast a')=\sigma(a)\star\sigma(a') \qquad \text{for all } a,a' \in [n],
	\end{equation}
	which we can, after defining $b\coloneqq \sigma(a)$ and $b'\coloneqq \sigma(a')$, equivalently write as
	\begin{equation}\label{eqn:iso'}
	b\star b' = \sigma(\sigma^{-1}(b)\ast\sigma^{-1}(b')) \qquad \text{for all } b,b' \in [m].
	\end{equation}
	In this case, we necessarily have $m=n$, and $\sigma \in S_n$ is a permutation of $[n]$. The reason for rewriting \eqref{eqn:iso} into \eqref{eqn:iso'} is that, given a magma $([n],\ast)$ and a permutation $\sigma \in S_n$, we can define a new operation on $[n]$ by
	\begin{equation}\label{eqn:iso_neu}
	b \ast_\sigma b' \coloneqq \sigma(\sigma^{-1}(b)\ast\sigma^{-1}(b')) \qquad \text{for all } b,b' \in [n].
	\end{equation}
	Then the corresponding Cayley table looks as follows:
	\[
	\begin{array}{c|ccc}
		\ast_\sigma & \sigma(1) & \cdots & \sigma(n) \\ \hline
		\sigma(1) & \sigma(1)\ast_\sigma\sigma(1) = \sigma(1\ast1) & \cdots & \sigma(1)\ast_\sigma\sigma(n) = \sigma(1\ast n) \\ 
		\vdots & \vdots & \ddots & \vdots \\
		\sigma(n) & \sigma(n)\ast_\sigma\sigma(1) = \sigma(n\ast1) & \cdots & \sigma(n)\ast_\sigma\sigma(n) = \sigma(n\ast n) \\
	\end{array}.
	\]
	It is immediately clear through the Cayley table that $\sigma$ is an isomorphism $([n],\ast) \to ([n],\ast_\sigma)$. Conversely, if $([n],\ast) \cong ([n],\star)$, then we see, by combining \eqref{eqn:iso'} and \eqref{eqn:iso_neu}, that there exists a permutation $\sigma\in S_n$ such that $\star = \ast_\sigma$. This means that, given a finite magma, we obtain an isomorphic magma by applying a permutation to all entries of the corresponding Cayley table, as seen in \Cref{ex:tables} with the cyclic permutation $\sigma = (1\:2)$; Bergeron et al.\ \cite{BLL} call such a procedure \emph{transport of structures}. On the other hand, this is the only way to obtain an isomorphic magma from a given finite magma on the same underlying set.
	
	Let $\Mag_n$ denote the set of all magmas on $[n]$. Hence, we can write the subset of $\Mag_n$ consisting of magmas that are isomorphic to $([n],\ast)$ as
	\begin{equation}\label{eqn:orbit}
	\{ ([n],\star) : ([n],\star) \cong ([n],\ast) \} = \{ ([n],\ast_\sigma) : \sigma \in S_n \}.
	\end{equation}
	We call it \emph{isomorphy class} of $([n],\ast)$ because isomorphic objects form an equivalence class (i.e., isomorphy is an equivalence relation); a likewise suitable term is \emph{isomorphy type}. Equivalence classes also arise in group actions, and we can see that the right-hand side of \eqref{eqn:orbit} looks like the orbit of $([n],\ast)$, provided that $S_n$ acts on $\Mag_n$ by $\sigma.([n],\ast) \coloneqq ([n],\ast_\sigma)$. This is indeed the case, which becomes clear via ${\id_{S_n}}. = \id_{\Mag_n}$ and
	\begin{align*}
	a\ \ast_{\sigma\circ\tau} \ a' &= \sigma\circ\tau((\sigma\circ\tau)^{-1}(a)\ast(\sigma\circ\tau)^{-1}(a')) \\ &= \sigma(\tau(\tau^{-1}(\sigma^{-1}(a))\ast\tau^{-1}(\sigma^{-1}(a')))) \\ &= \sigma(\sigma^{-1}(a) \ast_\tau \sigma^{-1}(a')) \\ &= a \ {(\ast_\tau)}_\sigma \ a'  \qquad \text{for all } a,a' \in [n] \text{ and } \sigma, \tau \in S_n,
	\end{align*}
	and thus, we get $(\sigma\circ\tau). = \sigma.\circ\tau.$ for all $\sigma, \tau \in S_n$. Our approach so far shows that each orbit of the action of $S_n$ just detected is nothing else than an isomorphy class of magmas on $[n]$; Bergeron et al.\ \cite{BLL} call such an action \emph{functoriality} of the transports of structures. As we are interested in computing the number of structurally different (i.e., non-isomorphic) magmas, we can thus try to find a way to compute the number of the respective orbits. The following lemma helps us.
	\begin{lemma}[{Cauchy--Frobenius--Burnside \cites{N}[Theorem 14.19]{J}}]\label{lem:orb}
		Given a finite group $G$ acting on a finite set $X$, the number of orbits is
		\[
		\frac1{\lvert G\rvert}  \sum_{g\in G} \lvert X^g\rvert,
		\]
		where $X^g \coloneqq \{x \in X : g.x=x\}$ denotes the fixed-point set of $g\in G$.
	\end{lemma}
	\begin{proof}
		Let $x \in X$, and let $G.x \coloneqq \{g.x : g\in G\}$ and $G_x \coloneqq \{g\in G : g.x=x\}$ denote the orbit and the stabilizer of $x$, respectively. By the orbit-stabilizer theorem \cite[Theorem 14.11]{J}, we have $\lvert G\rvert = \lvert G.x\rvert\lvert G_x\rvert$, and thus, we get
		\begin{align*}
			\lvert\{G.x : x \in X\}\rvert &= \sum_{x\in X} \frac1{\lvert G.x\rvert} \\&=  \frac1{\lvert G\rvert} \sum_{x\in X} \lvert G_x\rvert = \frac1{\lvert G\rvert}\lvert \{(g,x) \in G\times X : g.x=x\}\rvert = \frac1{\lvert G\rvert}  \sum_{g\in G} \lvert X^g\rvert.\qedhere
		\end{align*}
	\end{proof}	
	Transferred to our situation, where $G = S_n$ and $X = \Mag_n$, our goal is to compute $\lvert\Mag_n^\sigma\rvert$ for a given permutation $\sigma \in S_n$. First we see that
	\begin{align*}
		([n],\ast) \in \Mag_n^\sigma &\iff \ast_\sigma = \ast \\&\iff \forall b,b' \in [n] \colon b \ast_\sigma b' = b \ast b'  \\ &\iff \forall b,b' \in [n] \colon \sigma(\sigma^{-1}(b)\ast\sigma^{-1}(b')) = b \ast b'  \\ &\iff \forall a,a' \in [n] \colon \sigma(a\ast a') = \sigma(a) \ast \sigma(a') \\ &\iff \sigma \text{ is an automorphism of }([n],\ast).
	\end{align*}
	Let $\sigma \in S_n$ and $([n],\ast)\in \Mag_n^\sigma$. We arrange the rows and columns of the Cayley table of $\ast$ according to the cycle decomposition of $\sigma$, which means that we obtain a grid of rectangular sections.
	\begin{example}
		Let $n=5$ and $\sigma = (1\: 3\: 4)(2\: 5)$. Then a Cayley table arranged according to $\sigma$ has the shape
		\[
		\begin{array}{c|ccc:cc:}
			\ast & 1 & 3 & 4 & 2 & 5 \\ \hline
			1 & & & & & \\
			3 & & & & & \\
			4 & & & & & \\ \hdashline
			2 & & & & & \\
			5 & & & & & \\ \hdashline
		\end{array}.
		\]
	\end{example}
	Now consider a section of the Cayley table of $\ast$ of size $r \times s$:
	\begin{equation}\label{table_rs}
	\begin{array}{c|cccc}
		\ast  & b & \sigma(b) & \cdots & \sigma^{s-1}(b)  \\ \hline
		a  & x_{0,0} & x_{0,1} & \cdots & x_{0,s-1}  \\
		\sigma(a)  & x_{1,0} & x_{1,1} & \cdots & x_{1,s-1}  \\
		\vdots  & \vdots & \vdots & \ddots & \vdots  \\
		\sigma^{r-1}(a)  & x_{r-1,0} & x_{r-1,1} & \cdots & x_{r-1,s-1}  \\
	\end{array}.
	\end{equation}
	Because $\sigma$ is an automorphism, we have
	\[
	\sigma^k(x_{p,q})=\sigma^k(\sigma^p(a)\ast\sigma^q(b)) = \sigma^{k+p}(a) \ast \sigma^{k+q}(b) = x_{(k+p)\bmod r,\, (k+q)\bmod s}
	\]
	for all $k\in\mathbb Z$ and $(p,q)\in \{0,\dots,r-1\} \times \{0,\dots,s-1\}$, which implies
	\[
	\sigma^k(x_{p,q}) = x_{p,q} \iff x_{(k+p)\bmod r,\, (k+q)\bmod s} = x_{p,q} \iff r \mid k \ \land \ s \mid k \iff \lcm(r,s) \mid k.
	\]
	Thus, each entry $x_{p,q}$ already determines $\lcm(r,s)$ entries, which we can imagine as a diagonal chain directed downwards to the right starting and ending at $x_{p,q}$ (while jumping from one edge to the opposite possibly several times). Because we have $rs = \gcd(r,s)\lcm(r,s)$ entries, we obtain $\gcd(r,s)$ such chains within a rectangular section of size $r\times s$, i.e., we have $\gcd(r,s)$ independent choices of entries for a value assignment.
	
	In order to compute the number of possible value assignments, we first write the cycle type of $\sigma$ as the sequence $j(\sigma)\coloneqq(j_i(\sigma))_{i=1}^{\infty}$ meaning that $\sigma$, written as a composition of pairwise disjoint cycles, consists of $j_i(\sigma)$ cycles of length $i$, also called $i$-cycles. Clearly, we have $j_i(\sigma) = 0$ for all $i > n$ and $\sum_{i=1}^{n} ij_i(\sigma) = n$, which together is equivalent to $\sum_{i=1}^{\infty} ij_i(\sigma) = n$. Now, because of $\sigma^{\lcm(r,s)}(x_{p,q}) = x_{p,q}$, the value of $x_{p,q}$ must lie in a cycle of $\sigma$ whose length is a divisor of $\lcm(r,s)$. Hence, we have $\sum_{d\mid\lcm(r,s)} d j_d(\sigma)$ possible value assignments for each diagonal chain and thus
	\[\bigg(\sum_{d\mid\lcm(r,s)} d j_d(\sigma)\bigg)^{\gcd(r,s)}\]
	possible value assignments for a rectangular section of size $r \times s$. Taking all rectangular sections (according to the cycle type of $\sigma$) into account, we obtain
	\[
	\lvert\Mag_n^\sigma\rvert = \prod_{r,s=1}^n \bigg(\sum_{d\mid\lcm(r,s)} d j_d(\sigma)\bigg)^{\gcd(r,s)j_r(\sigma)j_s(\sigma)}.
	\]
	Note that this expression does not depend on the explicit form of $\sigma$, but only on its cycle type. Together with \Cref{lem:orb}, we finally obtain the number of isomorphy classes of finite magmas on a set with cardinality $n$:
	\begin{equation}\label{eqn:isoclass1}
	\frac{1}{n!}\sum_{\sigma\in S_n} \prod_{r,s=1}^n \bigg(\sum_{d\mid\lcm(r,s)} d j_d(\sigma)\bigg)^{\gcd(r,s)j_r(\sigma)j_s(\sigma)}.
	\end{equation}
	Due to the fact that there exist
	\[
	\frac{n!}{1^{j_1}\cdots n^{j_n} j_1! \cdots j_n!} = \frac{n!}{\prod_{i=1}^n i^{j_i} j_i!}
	\]
	permutations in $S_n$ whose cycle type is $(j_i)_{i=1}^{\infty}$, we can rewrite \eqref{eqn:isoclass1} into
	\begin{equation}\label{eqn:isoclass2}
	\sum_{j \in J_n} \frac{1}{\prod_{i=1}^n i^{j_i} j_i!} \prod_{r,s=1}^n \bigg(\sum_{d\mid\lcm(r,s)} d j_d\bigg)^{\gcd(r,s)j_rj_s},
	\end{equation}
	where $J_n$ denotes the set of all sequences $j=(j_i)_{i=1}^{\infty}$ of non-negative integers such that $\sum_{i=1}^{\infty} ij_i = n$. Thus, we have proved \Cref{thm:main}.
	
	\section{Generalizations}\label{sec:gen}
	Let us first recall how we started in \Cref{sec:pot}. Given a magma $([n],\ast)$ and a permutation $\sigma \in S_n$, we can write the operation $\ast_{\sigma}$ given by \eqref{eqn:iso_neu} also as
	\[
	\ast_{\sigma} = \sigma\circ\ast\circ(\sigma^{-1} \times \sigma^{-1}),
	\]
	where $f \times f$ denotes the componentwise application of a given map $f$ to a pair; clearly, we have $(\sigma \times \sigma)^{-1} = \sigma^{-1} \times \sigma^{-1}$. This means that the diagram
	\begin{equation*}
	\begin{tikzcd}
	{[n] \times [n]} & {[n]} \\
	{[n] \times [n]} & {[n]}
	\arrow["\ast", from=1-1, to=1-2]
	\arrow["{\sigma\times\sigma}"', from=1-1, to=2-1]
	\arrow["\sigma", from=1-2, to=2-2]
	\arrow["{\ast_{\sigma}}"', from=2-1, to=2-2]
	\end{tikzcd}
	\end{equation*}
	commutes, and thus, the permutation $\sigma$ is an isomorphism $([n],\ast) \to ([n],\ast_{\sigma})$. We generalize this phenomenon to operations with higher arity. For that, we introduce analogous notions for the respective algebraic structures.
	\begin{definition}
		Let $k$ be a non-negative integer. A \emph{$k$-magma} is a set $M$ equipped with a map $\ast \colon M^k \to M$, also called (internal $k$-ary) \emph{operation} on $M$, which we formally write together as a pair $(M,\ast)$.
	\end{definition}
	This means that $2$-magmas are magmas in the usual sense. Note that a $0$-magma is nothing else than a set together with a constant (from that set), and thus, the set must necessarily be non-empty. In general, for a finite set of $n$ elements, there are $n^{n^k}$ formally different operations in total, which we can imagine explicitly visualized as $k$-dimensional hypercubes with side length $n$ and thus $n^k$ entries. We get the following analogous definition.
	\begin{definition}
	Given two $k$-magmas $(M,\ast)$ and $(N,\star)$, a \emph{$k$-magma homomorphism} is a map $\phi \colon M \to N$ such that
	\begin{equation}\label{eqn:homomorphism}
	\phi(\ast(a_1,\dots,a_k))=\star(\phi(a_1),\dots,\phi(a_k)) \qquad \text{for all } a_1,\dots,a_k \in M.
	\end{equation}
	The remaining notions of \Cref{def:morphisms} are defined analogously.
	\end{definition}
	Let $f^{\times k} \coloneqq f \times \dots \times f$ denote the componentwise application of a given map $f$ to a $k$-tuple. Then we can write the homomorphism condition \eqref{eqn:homomorphism} also as
	\[
	\phi\circ\ast = \star\circ\phi^{\times k}.
	\]
	For composable maps $f$ and $g$, we clearly have $(f\circ g)^{\times k} = f^{\times k}\circ g^{\times k}$. Also, if $f$ is bijective, we clearly have $(f^{\times k})^{-1} = (f^{-1})^{\times k}$. Now, given a non-negative integer $k$, a $k$-magma $([n],\ast)$, and a permutation $\sigma \in S_n$, we analogously define
	\[
	\ast_{\sigma} \coloneqq \sigma\circ\ast\circ(\sigma^{-1})^{\times k}.
	\]
	This means that the diagram
	\begin{equation*}
		\begin{tikzcd}
			{[n]^k} & {[n]} \\
			{[n]^k} & {[n]}
			\arrow["\ast", from=1-1, to=1-2]
			\arrow["{\sigma^{\times k}}"', from=1-1, to=2-1]
			\arrow["\sigma", from=1-2, to=2-2]
			\arrow["{\ast_{\sigma}}"', from=2-1, to=2-2]
		\end{tikzcd}
	\end{equation*}
	commutes, and thus, the permutation $\sigma$ is an isomorphism $([n],\ast) \to ([n],\ast_{\sigma})$.
	
	Let $\kMag_n$ denote the set of all $k$-magmas on $[n]$. The next analogous step is to see again that $\sigma.([n],\ast) \coloneqq ([n],\ast_{\sigma})$ defines a group action of $S_n$ on $\kMag_n$. Indeed, we have ${\id_{S_n}}. = \id_{\kMag_n}$ and
	\begin{align*}
	\ast_{\sigma\circ\tau}
	&= \sigma\circ\tau\circ\ast\circ((\sigma\circ\tau)^{-1})^{\times k} \\
	&= \sigma\circ\tau\circ\ast\circ(\tau^{-1})^{\times k}\circ(\sigma^{-1})^{\times k}\\
	&= \sigma\circ\ast_{\tau}\circ(\sigma^{-1})^{\times k}\\
	&= {(\ast_{\tau})}_{\sigma} \qquad \text{for all }\sigma,\tau\in S_n,
	\end{align*}
	and thus, we get $(\sigma\circ\tau). = \sigma.\circ\tau.$ for all $\sigma, \tau \in S_n$. This means that we can apply \Cref{lem:orb}, and for that, we compute the cardinality of the fixed-point set $\kMag_n^\sigma$ for a given permutation $\sigma \in S_n$. We get an analogous equivalence:
	\begin{align*}
	([n],\ast) \in \kMag_n^\sigma &\iff \ast_{\sigma} = \ast \\
	&\iff \sigma\circ\ast\circ(\sigma^{-1})^{\times k} = \ast \\
	&\iff \sigma\circ\ast = \ast\circ\sigma^{\times k} \\
	&\iff \sigma \text{ is an automorphism of } ([n],\ast).
	\end{align*}
	Let $\sigma \in S_n$ and $([n],\ast)\in \kMag_n^\sigma$. We continue analogously as in \Cref{sec:pot} by arranging and dividing the respective $k$-dimensional operation hypercube of $\ast$ into hyperrectangular sections, i.e., $k$-dimensional boxes, according to the cycle type of $\sigma$. Considering such a box of size $r_1 \times \dots \times r_k$, we analogously obtain due to the automorphism property of $\sigma$ that each entry already determines $\lcm(r_1,\dots,r_k)$ entries, and thus, we have
	\[
	\frac{r_1\cdots r_k}{\lcm(r_1,\dots,r_k)}
	\]
	independent choices of entries for a value assignment. In order to be consistent, we set $\lcm(r_1,\dots,r_k) \coloneqq 1$ and $\lcm(r_1,\dots,r_k) \coloneqq r_1$ for the cases $k=0$ and $k=1$, respectively. Note that for $k>2$, we have
	\[
	\frac{r_1\cdots r_k}{\lcm(r_1,\dots,r_k)} \ne \gcd(r_1,\dots,r_k)
	\]
	in general, as we can see by the example $(r_1,\dots,r_k) = (2,\dots,2)$, and therefore, we cannot take the greatest common divisor as Harrison mistakenly did \cite[Theorem 5]{H}. The remaining analogous steps are to put the boxes together and eventually apply \Cref{lem:orb}. First we get
	\[
	\lvert\kMag_n^\sigma\rvert = \prod_{(r_1,\dots,r_k)\in [n]^k} \bigg(\sum_{d\mid\lcm(r_1,\dots,r_k)} d j_d(\sigma)\bigg)^{\frac{r_1\cdots r_k}{\lcm(r_1,\dots,r_k)}j_{r_1}(\sigma)\cdots j_{r_k}(\sigma)}.
	\]
	Now, by applying the lemma and using the same notation as in \eqref{eqn:isoclass2}, we obtain the following result.
	\begin{theorem}\label{thm:main_k}
	There are
	\begin{equation}\label{eqn:isoclass_k-magma}
	\sum_{j \in J_n} \frac{1}{\prod_{i=1}^n i^{j_i} j_i!} \prod_{(r_1,\dots,r_k)\in [n]^k} \bigg(\sum_{d\mid\lcm(r_1,\dots,r_k)} d j_d\bigg)^{\frac{r_1\cdots r_k}{\lcm(r_1,\dots,r_k)}j_{r_1}\cdots j_{r_k}}
	\end{equation}
	pairwise non-isomorphic $k$-magmas on a set with cardinality $n$, where $J_n$ denotes the set of all sequences $j=(j_i)_{i=1}^{\infty}$ of non-negative integers such that $\sum_{i=1}^{\infty} ij_i = n$.
	\end{theorem}
	
	\section{Using the cycle index}
	Riedel \cite{R} detected a helpful tool for practical computations of \eqref{eqn:isoclass2}, namely the cycle index, which is a formal polynomial whose monomials represent the cycle types that arise in a given permutation group, i.e., in a subgroup of the group $S_X$ of all permutations of a given finite set $X$. If a group $G$ acts faithfully on $X$, then we can think of $G$ as a permutation group, namely as the subgroup $G. \coloneqq \{g. : g \in G\}$ embedded in $S_X$. It is important to mention that the cycle types arising in this way depend not only on $G$ but in fact on its action on $X$. Analogously to the previous sections, we write $j(g.) = (j_i(g.))_{i=1}^{\infty}$ for the cycle type of the permutation $g.\in S_X$ where $g \in G$. Also, like in \Cref{thm:main,thm:main_k}, let $J_n$ denote the set of all sequences $j=(j_i)_{i=1}^{\infty}$ of non-negative integers such that $\sum_{i=1}^{\infty} ij_i = n$.
	\begin{definition}
		Given a finite group $G$ acting on a finite set $X$, the \emph{cycle index} of $G$ with respect to its action on $X$ is defined as the polynomial
		\[
		\frac1{\lvert G\rvert}\sum_{g\in G}\prod_{i=1}^{\lvert X\rvert} t_i^{j_i(g.)} \in \mathbb Q[t_1,\dots,t_{\lvert X\rvert}].
		\]
	\end{definition}
	This means that the monomial $\prod_{i=1}^{\lvert X\rvert} t_i^{j_i(g.)}$ represents the cycle type $(j_i(g.))_{i=1}^{\infty}$, and its coefficient indicates the corresponding relative frequency among all elements in $G.$. Transferred to our case, where $G=S_n$ plays the key role, we first obtain
	\begin{equation}\label{eqn:cycle_index}
	Z_n \coloneqq \frac{1}{n!}\sum_{\sigma\in S_n}  \prod_{i=1}^n t_i^{j_i(\sigma)} = \sum_{j\in J_n} \frac{1}{\prod_{i=1}^n i^{j_i} j_i!} \prod_{i=1}^n t_i^{j_i} \in \mathbb Q[t_1,\dots,t_n]
	\end{equation}
	as the cycle index of $S_n$ with respect to its natural action on $[n]$, which is given by $\sigma.a = \sigma(a)$.
	\begin{example}
		We have $Z_3 = \frac16(t_1^3+3t_1t_2+2t_3)$ because $S_3$ consists of one identity permutation ($[1^32^03^0]$), three transpositions ($[1^12^13^0]$), and two $3$-cycles ($[1^02^03^1]$).
	\end{example}
	A nice property and computational advantage of $Z_n$ is that it satisfies a recursion.
	\begin{lemma}[Pletsch \cite{P}]
		We have $Z_0 = 1$ and
		\[
		Z_n = \frac1n\sum_{i=1}^n t_i Z_{n-i} \qquad \text{for all }n>0.
		\]
	\end{lemma}
	\begin{proof}
		If $n=0$, then we have $\lvert S_n\rvert = \lvert J_n\rvert = 1$, so either sum in \eqref{eqn:cycle_index} runs over one element yielding the empty product $1$. If $n>0$, then we first note that, like $Z_n$, the polynomial $n!Z_n$ is also related to $S_n$ (with respect to the natural action), but with the difference that its coefficients indicate absolute frequencies instead of relative ones. Now, if we choose an element $m\in [n]$, then there are $\frac{(n-1)!}{(n-i)!}$ $i$-cycles in $S_n$ containing $m$. Each such cycle that appears in a permutation in $S_n$ contributes with a factor $t_i$ to the corresponding monomial of $Z_n$, while the remaining contribution comes from a cycle type of a permutation in $S_{n-i}$, i.e., a monomial of $Z_{n-i}$. Thus, the polynomial
		\[
		\frac{(n-1)!}{(n-i)!}t_i(n-i)!Z_{n-i} = (n-1)!t_iZ_{n-i},
		\]
		with absolute frequencies as coefficients, is related to the subset of $S_n$ consisting of permutations that have an $i$-cycle containing $m$. By summing this expression over all $i$ from $1$ to $n$, we cover all permutations in $S_n$ eventually because, given any such permutation, the element $m$ must lie in a cycle whose length is between $1$ and $n$. Thus, we obtain
		\[
		\sum_{i=1}^n (n-1)!t_iZ_{n-i} = n!Z_n,
		\]
		and dividing by $n!$ on both sides finishes the proof.
	\end{proof}
	
	Let us assume that we have computed a specific $Z_n$, possibly via the recursion given in the previous lemma. How can we use it to compute \eqref{eqn:isoclass2}? First we can map each monomial $\prod_{i=1}^n t_i^{j_i}$ of $Z_n$ to $\prod_{r,s=1}^n t_{\lcm(r,s)}^{\gcd(r,s)j_rj_s}$ and obtain
	\[
	Z_n^{[2]}\coloneqq \sum_{j\in J_n} \frac{1}{\prod_{i=1}^n i^{j_i} j_i!} \prod_{r,s=1}^n t_{\lcm(r,s)}^{\gcd(r,s)j_rj_s} \in \mathbb Q[t_1,\dots,t_{n^2}].
	\]
	In a similar way as in \Cref{sec:pot} starting from \eqref{table_rs}, we can see that $Z_n^{[2]}$ is the cycle index of $S_n$ with respect to its action on $[n]^2$ given by $\sigma.(a,a') = (\sigma(a),\sigma(a'))$, which looks technically the same as the automorphism property $\sigma(a\ast a') = \sigma(a)\ast\sigma(a')$ there. The main difference is that in the case of the action on $[n]^2$, one does not have to deal with value assignments, and hence, a chain of length $\lcm(r,s)$ leads to $t_{\lcm(r,s)}$.
	\begin{remark}
		For $n\le4$ and $n\ge4$, the highest index of an indeterminate appearing in $Z_n^{[2]}$ is bounded by $n$ and $\big\lfloor\frac{n^2}4\big\rfloor$, respectively.
	\end{remark}
	\begin{example}
		We have $Z_3^{[2]} = \frac16(t_1^9+3t_1t_2^4+2t_3^3)$ because the identity permutation, a transposition, and a $3$-cycle lead to a permutation fixing nine pairs, a permutation fixing one pair and having four transpositions of pairs, and a permutation consisting of three $3$-cycles of pairs, respectively.
	\end{example}
	Finally, by plugging $\sum_{d\mid\lcm(r,s)} d j_d$ into each $t_{\lcm(r,s)}$ of $Z_n^{[2]}$, we obtain \eqref{eqn:isoclass2}. Of course, we can generalize the method just described by defining
	\begin{equation*}
		Z_n^{[k]}\coloneqq \sum_{j\in J_n} \frac{1}{\prod_{i=1}^n i^{j_i} j_i!} \prod_{(r_1,\dots,r_k)\in [n]^k} t_{\lcm(r_1,\dots,r_k)}^{\frac{r_1\cdots r_k}{\lcm(r_1,\dots,r_k)}j_{r_1}\cdots j_{r_k}} \in \mathbb Q[t_1,\dots,t_{n^k}]
	\end{equation*}
	for all non-negative integers $k$, which, analogously, is the cycle index of $S_n$ with respect to its action on $[n]^k$ given by $\sigma. = \sigma^{\times k}$. Hence, by plugging $\sum_{d\mid\lcm(r_1,\dots,r_k)} d j_d$ into each $t_{\lcm(r_1,\dots,r_k)}$ of $Z_n^{[k]}$, we obtain \eqref{eqn:isoclass_k-magma}. Note that in the case $k=0$, we get $Z_n^{[k]} = t_1$ for all non-negative integers $n$, but the corresponding group action of $S_n$ is not faithful for $n>1$. This means that in order to plug in the respective values, we have to view $t_1$ still decomposed into summands that correspond to the cycle types of $S_n$. In the case $n=3$, for example, we have $Z_n^{[0]} = t_1 = \frac16(t_1+3t_1+2t_1)$, and now we can plug into each summand the corresponding $j_1$ (which eventually yields $1$).
	
	Finally, we present a suitable code in Sage and generate a few terms of the corresponding sequences for $k$-magmas in the cases $k \in \{0,\dots,4\}$:
	\begin{sageexample}
	sage: Pol.<t> = InfinitePolynomialRing(QQ)
	....: @cached_function
	....: def Z(n):
	....:     if n==0: return Pol.one()
	....:     return sum(t[k]*Z(n-k) for k in (1..n))/n
	....: def magmas(n,k): # number of isomorphy classes of k-magmas on a set of n elements
	....:     P = Z(n)
	....:     q = 0
	....:     coeffs = P.coefficients()
	....:     count = 0
	....:     for m in P.monomials():
	....:         p = 1
	....:         V = m.variables()
	....:         T = cartesian_product(k*[V])
	....:         for t in T:
	....:             r = [Pol.varname_key(str(u))[1] for u in t]
	....:             j = [m.degree(u) for u in t]
	....:             D = 0
	....:             lcm_r = lcm(r)
	....:             for d in divisors(lcm_r):
	....:                 try: D += d*m.degrees()[-d-1]    
	....:                 except: break
	....:             p *= D^(prod(r)/lcm_r*prod(j))
	....:         q += coeffs[count]*p
	....:         count += 1
	....:     return q
	sage: [magmas(n,0) for n in (0..10)]
	sage: [magmas(n,1) for n in (0..15)]   # A001372
	sage: [magmas(n,2) for n in (0..6)]    # A001329
	sage: [magmas(n,3) for n in (0..4)]    # incorrectly A001331, correctly A091510
	sage: [magmas(n,4) for n in (0..3)]
	\end{sageexample}
	In the case of a unary operation, we obtain sequence \seqnum{A001372}. If we think of the above sequences as rows of a matrix, then fixing the cardinality instead yields the corresponding columns:
	\begin{sageexample}
	sage: [magmas(0,k) for k in (0..10)]
	sage: [magmas(1,k) for k in (0..10)]
	sage: [magmas(2,k) for k in (0..7)]    # A191363 for k = 1,...,5
	sage: [magmas(3,k) for k in (0..4)]
	sage: [magmas(4,k) for k in (0..3)]
	\end{sageexample}
	In the case of cardinality $2$, we obtain a sequence whose $k$-th terms coincide with sequence \seqnum{A191363} for $k\in \{1,\dots,5\}$.

	\bigskip
	\hrule
	\bigskip

	\noindent 2020 {\it Mathematics Subject Classification}: Primary 08A62.

	\noindent \emph{Keywords:} magma, operation, arity, isomorphy class, group action, cycle type, cycle index

	\bigskip
	\hrule
	\bigskip

	\noindent (Concerned with sequences
	\seqnum{A001329},
	\seqnum{A001331},
	\seqnum{A091510},
	\seqnum{A001372}, and
	\seqnum{A191363}.)

\begin{thebibliography}{9}
		\bibitem{H} M. A. Harrison, The number of isomorphism types of finite algebras, {\it Proc. Amer. Math. Soc.} {\bf 17} (1966), 731--737.
		\bibitem{BLL} F. Bergeron, G. Labelle, and P. Leroux, {\it Combinatorial Species and Tree-like Structures}, Cambridge University Press, 1998.
		\bibitem{N} P. M. Neumann, A lemma that is not Burnside's, {\it Math. Sci.} {\bf 4} (1979), 133--141.
		\bibitem{J} T. W. Judson, {\it Abstract Algebra: Theory and Applications}, \url{http://abstract.ups.edu}, 2022.
		\bibitem{R} M. Riedel, How many non-isomorphic binary structures on the set of $n$ elements?, Mathematics Stack Exchange, \url{https://math.stackexchange.com/q/646925} (version: 2018-09-19)
		\bibitem{P} B. Pletsch, A recursion of the P\'olya polynomial for the symmetric group, {\it Math. Comput. Simulation} {\bf 80} (2010), 1212--1220.
	\end{thebibliography}
\end{document}